\documentclass{amsart}

\usepackage{amsmath}
\usepackage{amssymb}
\usepackage{caption}

\usepackage{xcolor}
\usepackage{graphicx}
\usepackage{nicefrac}
\usepackage{enumitem}
\setlist[enumerate]{label={\upshape(\roman*)}}
\usepackage[colorlinks=true, linkcolor=black, citecolor=black]{hyperref}
\usepackage{cleveref}

\newcommand{\norm}[1]{\ensuremath{\left\lVert #1 \right\rVert }}
\newcommand{\abs}[1]{\ensuremath{\left\lvert #1 \right\rvert}}
\newcommand{\angles}[1]{\ensuremath{\left\langle #1 \right\rangle}}

\newcommand{\snorm}[1]{\ensuremath{\lVert #1 \rVert }}

\expandafter\mathchardef\expandafter\varphi\number\expandafter\phi\expandafter\relax
\expandafter\mathchardef\expandafter\phi\number\varphi

\newtheorem{theorem}{Theorem}[section]
\newtheorem{corollary}[theorem]{Corollary}
\newtheorem{lemma}[theorem]{Lemma}
\newtheorem{proposition}[theorem]{Proposition}

\theoremstyle{definition}

\newtheorem{definition}[theorem]{Definition}
\newtheorem{example}[theorem]{Example}

\newtheorem*{openproblem}{Open Problem}

\numberwithin{equation}{section}

\DeclareMathOperator{\dist}{dist}

\DeclareMathOperator{\diag}{diag}
\DeclareMathOperator{\trace}{tr}
\DeclareMathOperator{\spans}{span}
\DeclareMathOperator{\conv}{conv}

\DeclareMathOperator{\ind}{ind}
\DeclareMathOperator{\Int}{Int}
\newcommand{\spec}{\ensuremath{\sigma}}
\newcommand{\essspec}{\ensuremath{\sigma_{\text{ess}}}}
\newcommand{\nr}{\ensuremath{W}}
\newcommand{\essnr}{\ensuremath{W_{\text{e}}}}
\newcommand{\Hil}{\ensuremath{\mathcal{H}}}

\newcommand{\D}{\ensuremath{\mathcal{D}}}
\newcommand{\uorbit}{\ensuremath{\mathcal{U}}}
\newcommand{\closure}[1]{\ensuremath{\overline{#1}}}
\renewcommand{\epsilon}{\varepsilon}

\DeclareFontFamily{U}{mathb}{\hyphenchar\font45}
\DeclareFontShape{U}{mathb}{m}{n}{
<-6> mathb5 <6-7> mathb6 <7-8> mathb7
<8-9> mathb8 <9-10> mathb9
<10-12> mathb10 <12-> mathb12
}{}
\DeclareSymbolFont{mathb}{U}{mathb}{m}{n}
\DeclareMathSymbol{\pprec}{\mathrel}{mathb}{"CE}
\DeclareMathSymbol{\ssucc}{\mathrel}{mathb}{"CF}

\begin{document}

\title[On diagonals of operators: selfadjoint, normal and other classes]{On diagonals of operators: \\ selfadjoint, normal and other classes}
\author{Jireh Loreaux}
\email{jloreau@siue.edu}
\address{Southern Illinois University Edwardsville \\
  Department of Mathematics and Statistics \\
  Edwardsville, IL, 62026-1653 \\
  USA}
\author{Gary Weiss$^{*}$}
\email{gary.weiss@uc.edu}
\address{University of Cincinnati   \\
  Department of Mathematics  \\
  Cincinnati, OH, 45221-0025 \\
  USA}
\thanks{$^{*}$The second named author is partially supported by Simons Foundation Collaboration Grant \#245014.}

\begin{abstract}
  We provide a survey of the current state of the study of diagonals of operators, especially selfadjoint operators.
  In addition, we provide a few new results made possible by recent work of M\"uller--Tomilov and Kaftal--Loreaux.
  This is an expansion of the second author's lecture part II at OT27.
\end{abstract}


\maketitle

\section{Introduction}
\label{sec:introduction}

By a diagonal of an operator $T \in B(\Hil)$ we mean a sequence $(\angles{Te_n,e_n})_{n=1}^{\infty}$ where $\{e_n\}_{n=1}^{\infty}$ is an orthonormal basis of $\Hil$.
The orthonormal basis is not fixed, and so $T$ has many diagonals.
Throughout this paper, we will use $\D(T)$ to denote the set of all diagonals of $T$.
Instead of varying the orthonormal basis, a useful equivalent viewpoint is to fix the orthonormal basis and consider the diagonals of the operators $UTU^{-1} = UTU^{*}$ in the unitary orbit $\uorbit(T)$ relative to this fixed orthonormal basis.
If $E$ denotes the canonical trace-preserving conditional expectation onto the subalgebra of diagonal operators determined by this fixed basis (i.e., $E$ denotes the operation of ``taking the main diagonal''), then there is a natural identification between $\D(T)$ and $E(\uorbit(T))$ via the *-isomorphism $\diag : \ell^{\infty} \to E(B(\Hil))$.
As such, sometimes we regard elements of $E(\uorbit(T)) \subseteq B(\Hil)$ as diagonals of $T$ even though they are operators as opposed to sequences.

The collection $\D(T)$ contains a substantial amount of information about the operator $T$.
For example, since every unit vector (in fact, any $k$-tuple of orthonormal vectors) is contained in some orthonormal basis, $\D(T)$ encodes the numerical range $\nr(T)$ (and correspondingly all the $k$-numerical ranges, see \cite{Hal-1964-ASM} for the origin of this notion).
Therefore, since an operator is selfadjoint if and only if its numerical range is contained in $\mathbb{R}$, it is clear that $T$ is selfadjoint if and only if $\D(T)$ contains only real-valued sequences.
This illustrates an example of how information about $\D(T)$ can yield obvious information about $T$.
A less obvious illustration of information encoding: when $\Hil$ is finite dimensional, $T$ is normal if and only if all the $k$-numerical ranges of $T$ are polygons (see \cite{Li-1994-LMA}), so from this, normality of $T$ can be determined from $\D(T)$.
We believe it is an open question whether normality of $T$ can be determined from $\D(T)$ when $\Hil$ is infinite dimensional; normality certainly cannot be determined solely from the $k$-numerical ranges of $T$.
Indeed, the latter is because, as is not hard to prove,  both the unilateral shift (which is non-normal) and a diagonal operator whose eigenvalues are dense in the open unit disk have that disk as their $k$-numerical range for each $k$.
Another less immediate example: diagonals also encode the essential numerical range by means of the fact that $\lambda \in \essnr(T)$ if and only if there is some diagonal of $T$ which contains a subsequence converging to $\lambda$.

Over the past century, diagonals of operators, especially of selfadjoint operators, have been investigated a great deal with substantial success.
A foundational result concerning diagonals of selfadjoint operators is due to Schur \cite{Sch-1923-SBMG} and Horn \cite[Theorems~1~and~5]{Hor-1954-AJM}, and is therefore called the \emph{Schur--Horn theorem}.
Of central importance in their theorem is the notion of \emph{majorization} of real-valued sequences.

\begin{definition}
  \label{def:majorization-finite}
  Given finite sequences $d,\lambda \in \mathbb{R}^n$, to say that $d$ is \emph{majorized by} $\lambda$ (or that $\lambda$ \emph{majorizes} $d$), denoted $d \prec \lambda$, means 
  \begin{equation}
    \label{eq:majorization-finite}
    \sum_{i=1}^m d^{*}_i \le \sum_{i=1}^m \lambda^{*}_i \quad\text{for}\ 1 \le m \le n,\ \text{and}\quad \sum_{i=1}^n d_i = \sum_{i=1}^n \lambda_i,
  \end{equation}
  where the sequences $d^{*}, \lambda^{*}$ denote the nonincreasing rearrangements of the sequences $d, \lambda$, respectively.
\end{definition}

\begin{theorem}[\cite{Sch-1923-SBMG,Hor-1954-AJM}]
  \label{thm:schur-horn}
  For a selfadjoint operator $T$ on a finite dimensional Hilbert space $\Hil$ of dimension $n$, with eigenvalue sequence $\lambda$, repeated according to multiplicity, and a sequence $d \in \mathbb{R}^n$, the following are equivalent:
  \begin{enumerate}
  \item \label{item:schur-horn-diagonal} $d$ is a diagonal of $T$ ($d \in \D(T)$);
  \item $d$ is majorized by $\lambda$ ($d \prec \lambda$);
  \item \label{item:schur-horn-convexity} $d$ is a convex combination of permutations of $\lambda$ ($d \in \conv \{ \lambda_{\pi} \in \mathbb{R}^n \mid \pi\ \text{is a permutation} \}$).
  \end{enumerate}
\end{theorem}

The Schur--Horn theorem is fascinating for several reasons.
Firstly, it was the first major result on diagonals of operators.
In addition, it provides a \emph{complete} characterization of the diagonals of any fixed selfadjoint operator $T$ on a finite dimensional Hilbert space.
Moreover, it shows that the diagonals $\D(T)$ of such an operator form a convex set which, as can be easily shown, has the permutations of the eigenvalue sequence as its extreme points.
This last fact is particularly surprising in that the authors are unaware of any direct proof; indeed, it is false in general if $T$ is not assumed to be selfadjoint (even for certain normal matrices, see \Cref{ex:nonconvex-normal-diagonal}).
Moreover, the unitary orbit of a selfadjoint operator (on a finite dimensional space) is completely determined by $\D(T)$ via any extreme point.
It is the Schur--Horn theorem that sparked a great deal of interest and focus on diagonals of \emph{selfadjoint} operators in particular which are the main subject of this survey.

The purpose of this paper is to provide a brief survey of the current state of knowledge, describe the history (apologies to the many significant unmentioned), add a few new results to the tapestry, and highlight open questions.
The information is arranged categorically rather than chronologically.
However, we try to provide some indication of the order in which results were discovered when they occur anachronistically in the order of appearance in the paper.

The remainder of this paper is organized as follows.
In \Cref{sec:compact-selfadjoint} we discuss results for diagonals of compact selfadjoint operators on an infinite dimensional Hilbert space.
These can be thought of in some way as the most direct generalizations of the Schur--Horn theorem.
Results from this section focus on work found in \cite{AK-2006-OTOAaA,GM-1964-MSN,KW-2010-JFA,LW-2015-JFA,Mar-1964-UMN}.
In \Cref{sec:finite-spectrum-selfadjoint} we review the study of diagonals of finite spectrum selfadjoint operators.
This was initiated by Kadison \cite{Kad-2002-PNASU,Kad-2002-PNASUa} and a complete characterization was provided by Bownik and Jasper \cite{BJ-2015-TAMS,BJ-2015-BPASM,Jas-2013-JFA}.
In \Cref{sec:general-selfadjoint} we discuss several results which hold for broad classes of selfadjoint operators coming from \cite{MT-2019-TAMS,Neu-1999-JFA}.
It is in this context that we are able to establish a new result which completely classifies diagonals of certain selfadjoint operators with at least three points in the essential spectrum (see \Cref{thm:diagonal-characterization}).
In \Cref{sec:normal-operators} we review the comparatively small amount of work that has been done for diagonals of normal operators due to \cite{Arv-2007-PNASU,Hor-1954-AJM,JLW-2016-IUMJ,Lor-2019-JOT,Wil-1971-JLMS2}.
Finally, in \Cref{sec:miscellaneous} we conclude with an overview of the few results which hold for more general classes of operators from \cite{Fan-1984-TAMS,FFH-1987-PAMS,Neu-1999-JFA,JLW-2016-IUMJ,Tho-1977-SJAM,Sin-1976-CMB,MT-2019-TAMS,Her-1991-RMJM}.

\section{Compact selfadjoint operators}
\label{sec:compact-selfadjoint}

We begin by defining some notation which occurs repeatedly throughout this section.
Let $c_0$ denote the collection of infinite sequences converging to zero, let $c_0^+$ its subset of nonnegative sequences, and let $c_0^{*}$ denote the subset of $c_0^+$ of nonincreasing sequences.
For a sequence $d \in c_0^+$ we let $d^{*} \in c_0^{*}$ denote the nonincreasing rearrangement\footnote{This is not technically a rearrangement in the sense that $d^{*}$ is not always a permutation of $d$, i.e., when $d$ has infinite support but is not strictly positive. However, this is in keeping with the standard terminology in the field and from measure theory.} of $d$.

Extending the Schur--Horn theorem to the setting of compact operators first requires a suitable notion of majorization.
It turns out that exactly which notion is appropriate depends on the context, but they agree on their common domain of definition (i.e., \Cref{def:weak-majorization,def:real-majorization} coincide for nonnegative sequences in $\ell^1$).

\begin{definition}[\protect{\cite[pp. 202--203]{GM-1964-MSN}}]
  \label{def:weak-majorization}
  Let $d,\lambda \in c_0^+$. 
  One says that $d$ is \emph{weakly majorized} by $\lambda$, denoted $d \pprec \lambda$, if for all $n \in \mathbb{N}$,
  \begin{equation*}
    \sum_{j=1}^n d^*_j \le \sum_{j=1}^n \lambda^*_j.
  \end{equation*}
  If in addition
  \begin{equation*}
    \sum_{j=1}^{\infty} d_j = \sum_{j=1}^{\infty} \lambda_j,
  \end{equation*}
  then $d$ is \emph{majorized} by $\lambda$, denoted $d \prec \lambda$; here we allow for the case when both sums are infinite.
\end{definition}

\begin{definition}[\protect{\cite[pp. 202--203]{GM-1964-MSN}}]
  \label{def:real-majorization}
  Let $d,\lambda \in \ell^1$ be real-valued. 
  One says that $d$ is \emph{majorized} by $\lambda$, denoted $d \prec \lambda$, if both $d_+ \pprec \lambda_+$ and $d_- \pprec \lambda_-$, and also $\sum_{j=1}^{\infty} d_j = \sum_{j=1}^{\infty} \lambda_j$.
\end{definition}

Gohberg and Markus \cite{GM-1964-MSN} were the first to extend the Schur--Horn theorem, and ultimately they characterized diagonals of selfadjoint trace-class operators modulo the number of zeros which occur in the diagonal sequence.

\begin{theorem}[\protect{\cite[Theorem 1]{GM-1964-MSN}}]
  \label{thm:gohberg-markus}
  Let $T$ be a selfadjoint trace-class operator in $B(\mathcal{H})$ with eigenvalue sequence $\lambda \in \ell^1$ repeated according to multiplicity. 
  Then
  \begin{enumerate}
  \item $d \in \D(T)$ implies $d \in \ell^1$ and $d \prec \lambda$, and conversely,
  \item $d \in \ell^1$ and $d \prec \lambda$ implies $d \oplus \mathbf{0} \in \D(T)$ for some dimension of $\mathbf{0}$.
  \end{enumerate}
\end{theorem}

The study of diagonals of selfadjoint operators then remained dormant for 35 years until A.~Neumann's generalization of the convexity portion of the Schur--Horn theorem in \cite{Neu-1999-JFA} (see \Cref{sec:general-selfadjoint} for details), although other results on diagonals of general operators did appear during this interim \cite{Fan-1984-TAMS,FF-1987-PRIASA,FFH-1987-PAMS,FF-1994-PAMS}.
However, it wasn't until the work of Kadison \cite{Kad-2002-PNASU,Kad-2002-PNASUa} in 2002, and of Arveson and Kadison \cite{AK-2006-OTOAaA} in 2006, that the study of diagonals was truly renewed.
This has sparked a flurry of activity that continues today;
\cite{Arv-2007-PNASU,Jas-2013-JFA,BJ-2014-CMB,BJ-2015-TAMS,BJ-2015-BPASM,Arg-2015-IEOT,Lor-2019-JOT} stemming from \cite{Kad-2002-PNASUa,Kad-2002-PNASU}, and \cite{KW-2010-JFA,LW-2015-JFA} from \cite{AK-2006-OTOAaA}.
The papers proceeding from \cite{Kad-2002-PNASU,Kad-2002-PNASUa} will be discussed in more detail in \Cref{sec:finite-spectrum-selfadjoint}.
For now, we will continue with a discussion of \cite{AK-2006-OTOAaA} and the ensuing papers.
Arveson and Kadison \cite{AK-2006-OTOAaA} restricted their attention to positive trace-class operators, albeit at the time of its writing, they appeared unaware of the trace-class work of Markus \cite{Mar-1964-UMN} and Gohberg--Markus \cite{GM-1964-MSN}.
The result Arveson and Kadison obtained (\Cref{thm:ak-schur-horn}) is closely related to \Cref{thm:gohberg-markus}, but they also stated related open problems in type II$_1$ factors which initiated study on this topic yet is outside the scope of this survey.
Forays into type II factors resulting from the impetus in \cite{Kad-2002-PNASUa,Kad-2002-PNASU,AK-2006-OTOAaA} can be found in \cite{AM-2007-IUMJ,AM-2008-JMAA,AM-2013-PJM,BR-2014-PAMS}

\begin{theorem}[\protect{\cite[Theorem 4.1]{AK-2006-OTOAaA}}]
  \label{thm:ak-schur-horn}
  Let $A \in \mathcal{L}^1_+$ be a positive trace-class operator.
  Then 
  \begin{equation*}
    E\Big(\overline{\mathcal{U}(A)}^{\snorm{\cdot}_1}\!\Big) = \{ \diag d \mid d \in \ell^1_+, d \prec s(A) \},
  \end{equation*}
  where $\norm{\cdot}_1$ denotes the trace norm, and $s(A)$ the singular value sequence.
\end{theorem}

The astute reader will have noticed that Arveson and Kadison considered the trace-norm closure of the unitary orbit instead of the unitary orbit itself.
The net effect of taking this closure is essentially to vary the size of the kernel of $A$, as Arveson and Kadison note in \cite[Proposition~3.1(iii)]{AK-2006-OTOAaA}.
It turns out that for a positive compact operator $A$, this effect can be achieved by a variety of constructions as we note in \Cref{prop:orbit-closure-equivalences}.

\begin{definition}[\protect{\cite[p.~3152]{KW-2010-JFA}}]
  \label{def:partial-isometry-and-singular-value-orbit}
  Let $A$ denote a positive compact operator with singular value sequence $s(A)$ and range projection $R_A$.
  The \emph{partial isometry orbit} $\mathcal{V}(A)$ is the collection
  \begin{equation*}
    \mathcal{V}(A) := \{ VAV^{*} \mid V^{*}V = R_A \}.
  \end{equation*}
  The \emph{singular value orbit} $S(A)$ is the collection of positive compact operators with the same singular values as $A$, namely,
  \begin{equation*}
    S(A) := \{ B \in \mathcal{K}_+ \mid s(B) = s(A) \}.
  \end{equation*}
\end{definition}

The next proposition appears in the first author's dissertation, and to our knowledge, is the only reference for this result.
When the operator is positive and compact, it unifies the seemingly disparate perspectives of the singular value orbit, partial isometry orbit, norm closure of the unitary orbit and, when the operator is trace-class, even the trace-norm closure of the unitary orbit.
This unification makes it possible to realize \Cref{thm:kwpartialisometryorbit} as a strict generalization of \Cref{thm:ak-schur-horn}.

\begin{proposition}[\protect{\cite[Proposition~2.1.12]{Lor-2016}}]
  \label{prop:orbit-closure-equivalences}
  If $A \in \mathcal{K}_+$ is a positive compact operator, then
  \begin{equation*}
    \mathcal{V}(A) = S(A) = \overline{\mathcal{U}(A)}^{\snorm{\cdot}}. 
  \end{equation*}
  If in addition $A$ is trace-class, then these are also equal to $\overline{\mathcal{U}(A)}^{\snorm{\cdot}_1}$. 
  Furthermore, if $A$ has finite rank, then all these sets coincide with the unitary orbit $\mathcal{U}(A)$.
\end{proposition}

The following fundamental result of Kaftal and Weiss \cite{KW-2010-JFA} characterizes the diagonals of positive compact operators in the partial isometry orbit, which, by \Cref{prop:orbit-closure-equivalences} is the same as the trace-norm closure of the unitary orbit when the operator is trace-class.
Therefore, their following \Cref{thm:kwpartialisometryorbit} is a generalization of \Cref{thm:ak-schur-horn} to compact operators.
Moreover, it has a striking resemblance to the majorization portion of the Schur--Horn theorem.

\begin{theorem}[\protect{\cite[Proposition 6.4]{KW-2010-JFA}}]
  \label{thm:kwpartialisometryorbit}
  Let $A\in \mathcal{K}_+$.
  Then
  \begin{equation*}
    E(\mathcal{V}(A)) = \{ \diag d \mid d \in c_0^+, d \prec s(A) \}.
  \end{equation*}
\end{theorem}

As we have already mentioned about all the results concerning compact operators thus far, they only characterize the diagonals \emph{modulo the dimension of the kernel}.
The next result, also from \cite{KW-2010-JFA}, is significant in that it overcomes this limitation, at least for positive compact operators with trivial kernel.
In the following, $\mathcal{D}$ denotes the diagonal operators.

\begin{theorem}[\protect{\cite[Proposition 6.6]{KW-2010-JFA}}] \label{thm:kwrangeprojectionidentity}
  Let $A\in \mathcal{K}_+$ whose range projection $R_A$ is the identity.
  Then 
  \begin{equation*}
    E(\mathcal{U}(A)) = E(\mathcal{V}(A))\cap\{ B\in\mathcal{D} \mid R_B=I \}.
  \end{equation*}
  From the equivalent viewpoint of diagonals as sequences, this becomes
  \begin{equation*}
    \D(A) = \{ d \in c_0^+ \mid d \prec s(A), d_n \not= 0 \ \text{for all}\ n \}.
  \end{equation*}
\end{theorem}

Since \cite{KW-2010-JFA}, it has become apparent to the authors and several other researchers (private communications) that understanding the interplay between the dimension of the kernel of a positive compact operator and its diagonal sequences is essential to characterizing diagonals of all selfadjoint operators more generally.
However, aside from the cases when the kernel is infinite dimensional or trivial, this remains an open problem.

\begin{openproblem}
  \label{prob:finite-kernel-schur-horn}
  Characterize, in terms of majorization, the diagonals of a positive compact operator with nontrivial, finite dimensional kernel.
  In particular, the following cases are important representative problems:
  \begin{enumerate}
  \item Characterize $\D\left(\diag\left(0,1,\frac{1}{2},\frac{1}{4},\frac{1}{8},\ldots\right)\right)$.
  \item Characterize $\D\left(\diag\left(0,1,\frac{1}{2},\frac{1}{3},\frac{1}{4},\ldots\right)\right)$.
  \end{enumerate}
\end{openproblem}

There has been some limited progress by the authors \cite{LW-2015-JFA} on the above problem, which we now describe.
In order to describe these results, we need the closely related notions of \emph{$p$-majorization} and \emph{approximate $p$-majorization}.

\begin{definition}
  \label{def:pmajorization}
  Given $d,\lambda\in c_0^+$ and $p \in \mathbb{Z}_{\ge 0}$, we say that $d$ is  \emph{$p$-majorized} by $\lambda$, denoted $d\prec_p\lambda$, if $d\prec\lambda$ and for sufficiently large $n$, we have the inequality
  \begin{equation*}
    \sum_{k=1}^{n+p} d^{*}_k \le \sum_{k=1}^{n} \lambda^{*}_k.
  \end{equation*}
  And $\infty$-majorization, denoted $d\prec_\infty\lambda$, means $d\prec_p\lambda$ for all $p\in\mathbb{N}$. 
\end{definition}

\begin{definition}
  \label{def:approximatepmajorization}
  Given $d,\lambda\in c_0^+$ and $p \in \mathbb{Z}_{\ge 0}$, we say that $d$ is \emph{approximately $p$-majorized} by $\lambda$, denoted $d\precsim_p\lambda$, if $d\prec\lambda$ and for every $\epsilon>0$, and for sufficiently large $n$, 
  \begin{equation*}
    \sum_{k=1}^{n+p} d^{*}_k \le  \sum_{k=1}^n \lambda^{*}_k + \epsilon \lambda^{*}_{n+1}.
  \end{equation*}
  Furthermore, if $d\precsim_p\lambda$ for infinitely many $p\in\mathbb{N}$ (equivalently obviously, for all $p\in\mathbb{N}$), this we call approximate $\infty$-majorization and denote it by $d\precsim_\infty\lambda$. 
\end{definition}

In the next theorem, for a positive compact operator $A$, $R_A^{\perp}$ is the projection onto the kernel of $A$, and so its trace is the dimension of the kernel.
Informally, this theorem says that for a sequence $d$:
\begin{enumerate}
\item If $d \prec_p s(A)$ and $d$ has $p$ fewer zeros than the dimension of the kernel of $A$, then $d \in \D(A)$.
\item If $d \in \D(A)$, then $d \precsim_p s(A)$ where $p$ is the difference in the number of zeros of $d$ and the dimension of the kernel of $A$.
\end{enumerate}

\begin{theorem}[\protect{\cite[Theorems~2.4 and 3.4]{LW-2015-JFA}}]\label{thm:sufficiencyofpmajorization}
  Let $A,B\in \mathcal{K}_+$,
  \begin{enumerate}
  \item \label{item:sufficiencyofpmajorization} If $B$ is a diagonal operator and for some $p \in \mathbb{Z}_{\ge 0} \cup \{\infty\}$,  $\trace R_B^\perp \le\trace R_A^\perp \le \trace R_B^\perp +p$ and $s(B)\prec_p s(A)$, then $B\in E(\mathcal{U}(A))$.
  \item \label{item:necessityofapproximatepmajorization} If $B\in E(\mathcal{U}(A))$, then $s(B) \precsim_p s(A)$, where
    \begin{equation*}
      p = \min \{ n\in \mathbb{Z}_{\ge 0} \cup \{\infty\} \mid \trace R_A^\perp \le\trace R_B^\perp +n \}
    \end{equation*}
  \end{enumerate}
\end{theorem}

The $p$-majorization condition in \Cref{thm:sufficiencyofpmajorization}\ref{item:sufficiencyofpmajorization} is known not to be a necessary condition for $d$ to be a diagonal of $A$ \cite[Example~2.6]{LW-2015-JFA}.
In contrast, it is not known whether approximate $p$-majorization in \Cref{thm:sufficiencyofpmajorization}\ref{item:necessityofapproximatepmajorization} is a sufficient condition for $d$ to be a diagonal of $A$.
However, since $\infty$-majorization and approximate $\infty$-majorization turn out to be equivalent concepts, there is the following corollary which characterizes diagonals of positive compact operators with infinite dimensional kernel.

\begin{corollary}[\protect{\cite[Corollary~3.5]{LW-2015-JFA}}]
  \label{cor:conjecturetrueforinfinitemo}
  Suppose $A\in \mathcal{K}_+$ has infinite rank and infinite dimensional kernel $(\trace R_A=\infty=\trace R_A^\perp)$.
  Then
  \[ E(\mathcal{U}(A)) = E(\mathcal{U}(A))_{fk}\sqcup E(\mathcal{U}(A))_{ik}, \]
  the members of $E(\mathcal{U}(A))$ with finite dimensional kernel and infinite dimensional kernel, respectively, are characterized by   
  \[ E(\mathcal{U}(A))_{fk}=\{ B\in\mathcal{D}\cap \mathcal{K}_+\mid s(B)\prec_\infty s(A)\quad\text{and}\quad \trace R_B^\perp<\infty\} \]
  and
  \[ E(\mathcal{U}(A))_{ik}=\{ B\in\mathcal{D}\cap \mathcal{K}_+ \mid s(B)\prec s(A)\quad\text{and}\quad \trace R_B^\perp=\infty\}. \]
\end{corollary}

Essentially, \Cref{thm:sufficiencyofpmajorization} says that if $A$ is a positive compact operator with infinite dimensional kernel, then $d \in \D(A)$ if and only if either (i) $d \prec s(A)$ and $d$ has infinitely many zeros, or (ii) $d \prec_{\infty} s(A)$.
Note that the infinite rank condition in \Cref{thm:sufficiencyofpmajorization} is not a true limitation, since the case when $A$ has finite rank was addressed as far back as \cite{AK-2006-OTOAaA} (because $\closure{\uorbit(A)}^{\norm{\cdot}_1} \!\! = \, \uorbit(A)$ by \Cref{prop:orbit-closure-equivalences}).

\section{Finite spectrum selfadjoint operators}
\label{sec:finite-spectrum-selfadjoint}

As we remarked in \Cref{sec:compact-selfadjoint}, Kadison authored two of the pioneering papers \cite{Kad-2002-PNASU,Kad-2002-PNASUa} which led to a resurgence in the study of diagonals of selfadjoint operators.
In these papers, Kadison investigated diagonals of projections starting from first principles, namely, the Pythagorean Theorem.
For the link between the Pythagorean Theorem and diagonals of projections, we refer the reader to Kadison's original paper \cite{Kad-2002-PNASU}.
However, the real surprise came in the second paper, where Kadison completely characterized diagonals of projections in $B(\Hil)$ with $\Hil$ separable and infinite dimensional, in which an unexpected integer appeared:

\begin{theorem}[\protect{\cite[Theorem~15]{Kad-2002-PNASUa}}]
  \label{thm:kadison-carpenter-pythagorean}
  A sequence $(d_n)$ is the diagonal of a projection $P$ if and only if it takes values in the unit interval and the quantities
  \begin{equation*}
    a := \sum_{d_n < \nicefrac{1}{2}} d_n \quad\text{and}\quad b := \sum_{d_n \ge \nicefrac{1}{2}} (1-d_n)
  \end{equation*}
  satisfy one of the mutually exclusive conditions
  \begin{enumerate}
  \item \label{item:a+b-infinite} $a+b = \infty$;
  \item \label{item:a+b-finite} $a+b < \infty$ and $a-b \in \mathbb{Z}$.
  \end{enumerate}
\end{theorem}

Since the advent of this theorem, there have been three primary outgrowths.
Firstly, there was a push to generalize Kadison's result to arbitrary selfadjoint operators with finite spectrum.
This is a natural extension because projections are just selfadjoint operators with two-point spectrum, suitably normalized.
Secondly, several authors have tried to explain the integer appearing in \Cref{thm:kadison-carpenter-pythagorean}.
Thirdly, Arveson found a generalization of this integer condition for certain normal operators with finite spectrum.
We will discuss the first two of these in this section, and the third in \Cref{sec:normal-operators}.
Of course, we would be remiss if we didn't mention that the first paper \cite{Kad-2002-PNASU} launched an investigation of ``diagonals'' (expectations of unitary orbits) in von Neumann factors (\cite{AM-2013-PJM,AM-2008-JMAA,AM-2007-IUMJ,BR-2014-PAMS}), but as mentioned in our introduction, this topic is outside the scope of this survey.

In a series of papers \cite{Jas-2013-JFA,BJ-2015-TAMS,BJ-2015-BPASM}, Bownik and Jasper managed to extend Kadison's theorem to arbitrary finite spectrum selfadjoint operators.
In \cite{Jas-2013-JFA}, Jasper handles the case when the selfadjoint operator has three points in the spectrum.
In \cite{BJ-2015-TAMS}, Bownik and Jasper do all the legwork to deal with the general case.
However, this results in a very complex theorem, in part because there is a great deal which depends on the precise multiplicity of each of the eigenvalues.
So, in \cite{BJ-2015-BPASM}, they provide a slightly simplified version of their theorem.
This new version, although still somewhat complex, is remarkably easier to state (see \Cref{thm:bownik-jasper-finite-spectrum} below), and it comes with an entirely independent and much shorter proof.

\begin{theorem}[\protect{\cite[Theorem~1.2]{BJ-2015-BPASM}}]
  \label{thm:bownik-jasper-finite-spectrum}
  Let $\{ \lambda_j \}_{j=0}^{n+1}$ be an increasing sequence of real numbers such that $\lambda_0 = 0$ and $\lambda_{n+1} = B$.
  Let $d$ be a sequence with values in $[0,B]$ such that $\sum d_j = \sum (B-d_j) = \infty$.
  For each $\alpha \in (0,B)$, define
  \begin{equation*}
    C(\alpha) = \sum_{d_j < \alpha} d_j \quad\text{and}\quad D(\alpha) = \sum_{d_j \ge \alpha} (B - d_j).
  \end{equation*}
  There exists a selfadjoint operator $T$ with finite spectrum $\spec(T) = \{\lambda_j\}_{j=0}^{n+1}$ and diagonal $d$ if and only if either
  \begin{enumerate}
  \item $C(B/2) + D(B/2) = \infty$, or 
  \item $C(B/2) + D(B/2) < \infty$ and there exist $N_1, \ldots, N_n \in \mathbb{N}$ and $k \in \mathbb{Z}$ such that
    \begin{equation*}
      C(B/2) - D(B/2) = \sum_{j=1}^n \lambda_j N_j + k B
    \end{equation*}
    and for all $1 \le r \le n$,
    \begin{equation*}
      (B-\lambda_r) C(\lambda_r) + \lambda_r D(\lambda_r) \ge (B - \lambda_r) \sum_{j=1}^r \lambda_j N_j + \lambda_r \sum_{j=r+1}^n (B-\lambda_j)N_j.
    \end{equation*}
  \end{enumerate}
\end{theorem}

Note that \Cref{thm:bownik-jasper-finite-spectrum} does not specify the precise multiplicity of the eigenvalues.
This is not a deficiency of the theorem, but rather a feature; Bownik--Jasper have a theorem which completely characterizes diagonals of a finite spectrum selfadjoint operator with specified multiplicities of the eigenvalues \cite[Theorem~1.3]{BJ-2015-TAMS}, but that theorem is significantly more cumbersome.

In addition to their generalizations of \Cref{thm:kadison-carpenter-pythagorean}, Bownik and Jasper also provided a new and different proof of the sufficiency direction of this theorem \cite{BJ-2014-CMB}, which Kadison referred to as the Carpenter's theorem, i.e., that conditions \ref{item:a+b-infinite} and \ref{item:a+b-finite} of \Cref{thm:kadison-carpenter-pythagorean} imply that $d$ is a diagonal of a projection $P$.
This new proof was constructive, and as a byproduct ensured that the theorem was true even for \emph{real} Hilbert spaces, not just complex ones.
While that may seem like an esoteric distinction, this is a topic that has actually arisen repeatedly in the study of diagonals of selfadjoint operators, even as far back as Horn's original paper \cite{Hor-1954-AJM} (see also \cite{KW-2010-JFA} in their discussion of orthogonal matrices, i.e., unitaries with real entries; or see \cite{JLW-2016-IUMJ}).

The other primary outgrowth of \Cref{thm:kadison-carpenter-pythagorean} is the elucidation of the necessity direction (which Kadison referred to as the Pythagorean Theorem), particularly the integer in condition \Cref{thm:kadison-carpenter-pythagorean}\ref{item:a+b-finite}.
Even Kadison referred to it as ``the curious `integrality' condition imposed on $a-b$'' \cite[p.~5220]{Kad-2002-PNASUa}.
Perhaps more surprising is Kadison's proof, which concludes: ``As $a-b$ is arbitrarily close to an integer, $a-b$ is an integer'' \cite[p.~5221]{Kad-2002-PNASUa}.
This is a rather analytic way to prove some quantity is an integer, and in this case the proof is rather opaque and does not lend much insight into the origin of this integer.

As a result of this unexplained integer, several authors have given new proofs of the necessity of \Cref{thm:kadison-carpenter-pythagorean}\ref{item:a+b-finite}.
First, in \cite{Arv-2007-PNASU} Arveson recognized the integer as the index of a certain Fredholm operator, but that description lacked a natural explanation of the role of this Fredholm operator.
Later, Kaftal, Ng and Zhang gave another independent proof \cite[Corollary~3.6]{KNZ-2009-JFA} of the necessity of \Cref{thm:kadison-carpenter-pythagorean}\ref{item:a+b-finite} while working on a seemingly unrelated topic: strong sums of projections in von Neumann factors.
More recently, Argerami provided yet another proof [Argerami, Theorem~4.6] based on an argument of Effros [Effros, Lemma~4.1].
However, the next theorem, due jointly to the first author and V.~Kaftal, provides a very direct path from the original projection $P$ to the integer $a-b$ via the notion of essential codimension of projections due to Brown, Douglas and Fillmore \cite[Remark~4.9]{BDF-1973-PoaCoOT}.

\begin{definition}[\protect{\cite[Remark~4.9]{BDF-1973-PoaCoOT}}]
  \label{def:essential-codimension}
  For projections $P,Q$ with $P-Q$ compact, the \emph{essential codimension} of $Q$ in $P$, denoted $[P:Q]$, is the integer defined by
  \begin{equation*}
    [P:Q] :=
    \begin{cases}
      \trace P-\trace Q & \text{if}\ \trace P\ \text{and}\ \trace Q < \infty, \\[0.5em]
      \ind(V^{*}W) & \parbox[c][2em]{0.5\textwidth}{if $\trace P = \trace Q = \infty$, where \\
       $W^{*}W = V^{*}V = I, WW^{*} = P, VV^{*} = Q$.} \\[0.4em]
    \end{cases}
  \end{equation*}
  An equivalent alternative definition is $[P:Q] := \ind(QP)$ where $QP : P\Hil \to Q\Hil$.
\end{definition}

\begin{theorem}[\protect{\cite[Theorem~1.3]{KL-2017-IEOT}}]
  \label{thm:kadison-integer-essential-codimension}
  With the notations of \Cref{thm:kadison-carpenter-pythagorean}, if $P \in B(\Hil)$ is a projection with $a + b < \infty$ and $Q$ is the projection onto $\spans \{ e_j \mid d_j \ge \nicefrac{1}{2} \}$, then $P-Q$ is Hilbert--Schmidt and $a - b = [P:Q]$ is the essential codimension of the pair $P,Q$.
\end{theorem}

\Cref{thm:kadison-integer-essential-codimension} can also illuminate the role of the integer $k$ in \Cref{thm:bownik-jasper-finite-spectrum}, but the details are more technical (see \cite[Section~3]{KL-2017-IEOT}).
Moreover, there is a collection of integers in Arveson's generalization of \Cref{thm:kadison-carpenter-pythagorean} to finite spectrum normal operators which can also be explained in a similar way via the notion of essential codimension (see \Cref{sec:normal-operators} herein, especially \Cref{thm:arveson,thm:loreaux-normal-essential-codimension}, for details).

\section{General selfadjoint operators}
\label{sec:general-selfadjoint}

In this section we present both a survey and new results on this topic.

\subsection{Existing results}

In \Cref{sec:compact-selfadjoint,sec:finite-spectrum-selfadjoint} we considered two extensions of the Schur--Horn theorem to infinite dimensions, both of which were conservative in their spectral characteristics;
indeed, there were only ever finitely many points in the essential spectrum, with the interesting cases being when there were only one or two.
In this section, we explore results about diagonals of selfadjoint operators which have much weaker, or even no constraints on the spectral characteristics of the operator.
The first of these is due to Neumann in \cite[Theorem~3.13]{Neu-1999-JFA}, where he provided a generalization of the convexity aspect of the Schur--Horn theorem to diagonalizable selfadjoint operators.

\begin{theorem}[\protect{\cite[Theorem~3.13]{Neu-1999-JFA}}]
  \label{thm:neumann-diagonalizable-convexity}
  If $T = \diag(\lambda)$ is a diagonal selfadjoint operator on $B(\Hil)$ with eigenvalue sequence $\lambda$, then
  \begin{equation*}
    \closure{\D(T)}^{\norm{\cdot}_{\infty}} = \closure{\conv \{ \lambda_{\pi} \mid \pi\ \text{is a permutation} \}}^{\norm{\cdot}_{\infty}}.
  \end{equation*}
\end{theorem}

While this theorem is incredibly interesting as a generalization of the Schur--Horn theorem, we remark that because it takes closures, it loses much of the fine structure present in the theorems mentioned in \Cref{sec:compact-selfadjoint,sec:finite-spectrum-selfadjoint}.
For example, when applied to an infinite, coinfinite projection $P$, we can see that $\closure{\D(P)}^{\norm{\cdot}_{\infty}}$ consists precisely of those sequences with values in $[0,1]$, thereby masking the subtle integer condition present in the characterization of $\D(P)$ in \Cref{thm:kadison-carpenter-pythagorean}\ref{item:a+b-finite}.

In the case of \Cref{thm:neumann-diagonalizable-convexity} when $T$ is a positive compact operator, so that $\lambda \in c_0^+$, then the set on the right-hand side is equal to $\{ d \in c_0^+ \mid d \pprec \lambda \}$ (see \cite[Corollary~2.18]{Neu-1999-JFA}).
Contrasting this to the more nuanced (i.e., exact) \Cref{thm:kwpartialisometryorbit}, we see that the effect of taking the $\ell^{\infty}$-closure is in some sense to ignore the exact value of the trace (anything smaller will do).
Therefore, even when restricted to compact, or even trace-class, operators, \Cref{thm:neumann-diagonalizable-convexity} misses out on some fine structure in the set of diagonal sequences.
Nevertheless, this is an important result in the understanding of diagonals of selfadjoint operators because of its generality.

In addition, Neumann proved another powerful result \cite[Remark~4.5]{Neu-1999-JFA}, again using the $\ell^{\infty}$-closure, which applies also to \emph{all} selfadjoint operators.
In essence, this theorem states that a sequence $d$ is the diagonal of a selfadjoint operator $T$ if and only if the part of $d$ which lies outside the convex hull of the essential spectrum of $T$ is majorized by the part of the spectrum which lies outside the convex hull of the essential spectrum.

\begin{theorem}[\protect{\cite[Remark~4.5]{Neu-1999-JFA}}]
  \label{thm:neumann-general-schur-horn}
  Suppose that $T \in B(\Hil)$ is selfadjoint and let $\alpha_- = \min \essspec(T)$ and $\alpha_+ = \max \essspec(T)$.
  Let $T^{\pm} := (T-\alpha_{\pm})_{\pm}$, which are both positive compact operators.
  Let $d \in \ell^{\infty}$ be a real-valued sequence.
  Then $d \in \closure{\D(T)}^{\norm{\cdot}_{\infty}}$ if and only if there are sequence $d^{\pm} \in c_0^+$ and $d'$ with values in $[\alpha_-,\alpha_+]$ such that $d = d' + d^+ - d^-$ and $d^{\pm} \pprec s(T^{\pm})$.
\end{theorem}

Very recently, M\"uller and Tomilov \cite{MT-2019-TAMS} recognized an important pattern in results about diagonals of operators, and they have turned this into a powerful theorem.
In particular, they noticed that often a sufficient condition for a sequence to be the diagonal of some operator is that the sequence lies ``well-inside'' the interior of the essential numerical range, in that it does not approach the boundary too rapidly.
This involves what they refer to as a Blaschke-type condition (see \Cref{thm:blaschke-muller-tomilov} below) in reference to the Blaschke product analytic function on the open unit disk, albeit in \Cref{thm:blaschke-muller-tomilov} the condition is that the sum is infinite rather than finite.
Situations in which M\"uller and Tomilov noticed the Blaschke-type condition prior to their discovery include \Cref{thm:kadison-carpenter-pythagorean}\ref{item:a+b-infinite}, \Cref{thm:jasper-loreaux-weiss-unitary}, and other examples due to Herrero \cite{Her-1991-RMJM}.
The power of their theorem is at least two-fold: it applies to the vast majority of selfadjoint operators, and it provides a sufficient condition for a sequence to lie in $\D(T)$, as opposed to the closure of this set as in Neumann's theorems mentioned above.

\begin{theorem}[\protect{\cite[Theorem~1.1]{MT-2019-TAMS}}]
  \label{thm:blaschke-muller-tomilov}
  Let $T \in B(\Hil)$ be selfadjoint and let $d = (d_k)_{k=1}^{\infty} \subseteq \Int_{\mathbb{R}} \essnr(T)$ satisfy the Blaschke condition
  \begin{equation}
    \label{eq:blaschke-condition}
    \sum_{k=1}^{\infty} \dist \{ d_k, \mathbb{R} \setminus \essnr(T) \} = \infty.
  \end{equation}
  Then $d \in \mathcal{D}(T)$.
\end{theorem}

We remark that M\"uller and Tomilov's entire paper actually applies more generally to finite tuples of operators, but we have restricted here our focus to the setting of a single operator both because it is more in line with the scope of this paper and because it would be cumbersome to define the joint essential numerical range.

\Cref{thm:blaschke-muller-tomilov} is a wonderful \emph{tour de force} for establishing that such sequences occur as diagonals and it subsumes in part several earlier results of others (e.g., \Cref{thm:kadison-carpenter-pythagorean,thm:bownik-jasper-finite-spectrum,thm:neumann-general-schur-horn}).
However, the parts that it misses are the edge cases (e.g., \Cref{thm:kadison-carpenter-pythagorean}\ref{item:a+b-finite}), and these are in general the harder results to establish.
Moreover, because this is concerned with the interior of the essential numerical range, this theorem has nothing to say about diagonals of compact operators, whose essential numerical range is zero.
In this sense, \Cref{thm:blaschke-muller-tomilov} is orthogonal to the study of diagonals of compact operators explored in \Cref{sec:compact-selfadjoint}.

\subsection{New results}

In what follows, we use the above result (\Cref{thm:blaschke-muller-tomilov}) of M\"uller and Tomilov to provide the legwork in establishing that certain sequences are diagonals of selfadjoint operators.
In conjunction, we use \Cref{lem:kaftal-loreaux-restriction-lemma} from \cite[Lemma~3.3]{KL-2017-IEOT} to place a constraint (\Cref{cor:kaftal-loreaux-corollary}) on which sequences can appear as the diagonals of certain selfadjoint operators.
Consequently, in \Cref{thm:diagonal-characterization} we characterize the diagonals of all selfadjoint operators whose numerical range is contained in the essential numerical range (this is equivalent to the minimum and maximum of the spectrum being elements of the essential spectrum), as long as there are at least three points in the essential spectrum.
This class of operators includes, for example, all selfadjoint multiplication operators on a nonatomic measure space.

We begin with the lemma from \cite{KL-2017-IEOT} which has as a consequence a constraint on the essential spectrum of a selfadjoint operator when the Blaschke condition \eqref{eq:blaschke-condition} is finite instead of infinite.

\begin{lemma}[\protect{\cite[Lemma~3.3]{KL-2017-IEOT}}]
  \label{lem:kaftal-loreaux-restriction-lemma}
  let $\mathcal{J}$ be a proper ideal, $T \in B(\Hil)_+$ a positive contraction, and $Q \in B(\Hil)$ a projection.
  \begin{enumerate}
  \item If $Q - QTQ \in \mathcal{J}$ and $Q^{\perp}TQ^{\perp} \in \mathcal{J}$, then $T - Q \in \mathcal{J}^{\nicefrac{1}{2}}$ and $T\chi_{[0,\epsilon]}(T) \in \mathcal{J}$ for every $0 < \epsilon < 1$. ($\chi_{[0,\epsilon]}(T)$ is the spectral projection onto $[0,\epsilon]$.)
  \item Assume that $T$ is a projection or that $\mathcal{J}$ is idempotent (i.e., $\mathcal{J} = \mathcal{J}^2$).
    If $T - Q \in \mathcal{J}^{\nicefrac{1}{2}}$ and $T\chi_{[0,\epsilon]}(T) \in \mathcal{J}$ for some $0 < \epsilon < 1$, then $Q - QTQ \in \mathcal{J}$ and $Q^{\perp}TQ^{\perp} \in \mathcal{J}$.
  \end{enumerate}
\end{lemma}

While it may not be immediately obvious, \Cref{lem:kaftal-loreaux-restriction-lemma} has the following new corollary.

\begin{corollary}
  \label{cor:kaftal-loreaux-corollary}
  Suppose that $T$ is a selfadjoint operator in $B(\Hil)$ with some diagonal $d \in \D(T)$.
  Let $a = \min \essspec(T)$ and $b = \max \essspec(T)$.
  If $(T-b)_+$ and $(T-a)_-$ are trace-class and
  \begin{equation}
    \label{eq:blaschke-kaftal-loreaux}
    \sum_{k=1}^{\infty} \dist \{ d_k, \{a,b\} \} < \infty,
  \end{equation}
  then the essential spectrum $\essspec(T)$ contains at most two points.
\end{corollary}

\begin{proof}[Proof of corollary]
  If $a = b$, then $\essspec(T) = \{a\}$ and there is nothing to prove, so suppose $a < b$.
  
  We first reduce to the case when $a = 0$ and $b = 1$.
  In order to do this, we simply replace $T$ with $\frac{1}{b-a}(T-a)$.
  Note that because this is just a scaling and translation of $T$, the cardinality of the essential spectrum is preserved.
  Moreover, the diagonal of this new operator is the sequence $\big(\frac{d_k-a}{b-a}\big)_{k=1}^{\infty}$,
  and it is straightforward to check that the corresponding sum \eqref{eq:blaschke-kaftal-loreaux} differs from the one for $T$ by a factor of $b-a$.

  From the preceding paragraph, we may assume without loss of generality that $a = 0$ and $b = 1$.
  We next reduce to the case when $(T-1)_+$ and $T_-$ are zero.
  For this, simply replace $T$ with $T' = T - (T-1)_+ + T_-$.
  Then notice that their difference $T - T'$ is a trace-class operator, and hence $T,T'$ have the same essential spectrum.
  The former fact implies that the difference in their diagonal sequences $(d'_k - d_k)_{k=1}^{\infty}$ is absolutely summable.
  Moreover,
  \begin{equation*}
    \sum_{k=1}^{\infty} \dist \{ d'_k, \{a,b\} \} \le \sum_{k=1}^{\infty} \big(\abs{d'_k - d_k} + \dist \{ d_k, \{a,b\} \}\big) < \infty.
  \end{equation*}
  Hence, the hypotheses of the theorem are satisfied by $T'$ and this has the same essential spectrum as $T$, so if \Cref{cor:kaftal-loreaux-corollary} holds for $T'$ it also holds for $T$.

  From the preceding paragraph, we may assume without loss of generality that $(T-1)_+$ and $T_-$ are zero, so that $T$ is a positive contraction.
  This implies that $\nr(T) \subseteq \essnr(T) = [0,1]$; note that we don't necessarily have equality with $\nr(T)$ because $\nr(T)$ may not contain the endpoints $0,1$.
  In particular, this means that $(d_k)_{k=1}^{\infty} \subseteq [0,1]$.
  Therefore, the sum \eqref{eq:blaschke-condition} can be split and written as
  \begin{equation*}
    \sum_{d_k < \nicefrac{1}{2}} d_k + \sum_{d_k \ge \nicefrac{1}{2}} (1-d_k) = \sum_{k=1}^{\infty} \dist \{ d_k, \{0,1\} \} < \infty.
  \end{equation*}
  So, in particular, both sums on the left are finite.
  Let $(e_k)_{k=1}^{\infty}$ be the orthonormal basis in which $T$ has the diagonal $(d_k)_{k=1}^{\infty}$.
  Then let $Q$ denote the projection onto the closed span of $\{ e_k \mid d_k \ge \nicefrac{1}{2} \}$.
  Since $T$ is a positive contraction, $Q^{\perp}TQ^{\perp}$ and $Q-QTQ = Q(1-T)Q$ are positive operators.
  Moreover, we have
  \begin{equation*}
    \trace(Q^{\perp}TQ^{\perp}) = \sum_{d_k < \nicefrac{1}{2}} d_k, \qquad \trace(Q-QTQ) = \sum_{d_k \ge \nicefrac{1}{2}} (1-d_k).
  \end{equation*}
  Since both of these sums are finite and the operators are positive, $Q^{\perp}TQ^{\perp}$ and $Q-QTQ$ are trace-class.
  Therefore $T,Q$ satisfy the hypotheses of \Cref{lem:kaftal-loreaux-restriction-lemma}(i) when $\mathcal{J}$ is the ideal of trace-class operators, and hence $T-Q$ is Hilbert--Schmidt.
  Being compact, $T-Q$ has image zero in the Calkin algebra.
  Therefore, the image of $T = (T-Q) + Q$ is a projection in the Calkin algebra and hence $T$ has at most $\{0,1\}$ in its essential spectrum.
\end{proof}

Using \Cref{thm:blaschke-muller-tomilov} and \Cref{cor:kaftal-loreaux-corollary} we can completely characterize the diagonals of a large class of selfadjoint operators with at least three points in their essential spectrum.
This class includes all selfadjoint multiplication operators on a nonatomic measure space.
In addition, every selfadjoint operator is a compact perturbation of an operator from this class, hence the image of this class in the Calkin algebra consists of all the selfadjoint elements.

\begin{theorem}
  \label{thm:diagonal-characterization}
  Suppose that $T \in B(\Hil)$ is a selfadjoint operator with $\nr(T) \subseteq \essnr(T)$ and with at least three points in the essential spectrum.
  Then a sequence $(d_k)_{k=1}^{\infty}$ lies in $\mathcal{D}(T)$ if and only if $(d_k)_{k=1}^{\infty} \subseteq \nr(T)$ and
  \begin{equation*}
    \sum_{k=1}^{\infty} \dist \{ d_k, \mathbb{R} \setminus \essnr(T) \} = \infty,
  \end{equation*}
  and the number of occurrences of each of $\min \essspec(T), \max \essspec(T)$ in the sequence $(d_k)_{k=1}^{\infty}$ is less than or equal to the dimension of the corresponding eigenspace of $T$.
\end{theorem}

\begin{proof}
  Before we prove either direction, we establish a few facts relevant to both directions.
  When $T$ is normal it is well-known that $\closure{\nr(T)} = \conv \spec(T)$ and the closed set $\essnr(T) = \conv \essspec(T)$.
  Then, since $T$ is selfadjoint, if $a := \min\essspec(T)$ and $b := \max\essspec(T)$, then $\conv \essspec(T) = [a,b]$ and so $\closure{W(T)} = \conv \spec(T) \supseteq \conv \essspec(T) = \essnr(T) = [a,b]$. Moreover, by hypothesis $\nr(T) \subseteq \essnr(T)$, and using the previous string of inclusions we conclude $\closure{\nr(T)} = \essnr(T) = [a,b]$.
  Finally, since $\nr(T)$ is convex and hence an interval that contains $(a,b)$, we know that the difference $\essnr(T) \setminus \nr(T)$ is contained in $\{a,b\}$.
  Additionally, $(T-a)_- = (T-b)_+ = 0$ since $a,b$ are also the minimum and maximum of $\spec(T)$.

  We first prove the ``only if'' direction.
  Suppose that $(d_k)_{k=1}^{\infty} \in \mathcal{D}(T)$.
  First note that $d_k = \angles{Te_k,e_k} \in \nr(T)$.
  Next, because $T$ has at least three points in its essential spectrum by hypothesis, and $(T-a)_- = (T-b)_+ = 0$ since $\nr(T) \subseteq \essnr(T)$, we can use \Cref{cor:kaftal-loreaux-corollary} to conclude that $\sum_{k=1}^{\infty} \dist \{ d_k, \mathbb{R} \setminus \essnr(T) \} = \infty.$
  Finally, it is elementary to show that $a$ (or $b$) $\in \nr(T)$ if and only if $a$ (or $b$) is an eigenvalue of $T$ because these points are the minimum and maximum of $\spec(T)$.
  In fact, more is true; if $\angles{Te_k,e_k} = a$ (or $b$), then $e_k$ is an eigenvector for $a$ (or $b$).
  Therefore, if a diagonal sequence $(d_k)_{k=1}^{\infty}$ of $T$ takes the value $a$ exactly $m$ times (here, $m = \infty$ is allowable), then there are at least $m$ orthogonal eigenvectors for $a$, so the dimension of the eigenspace is at least $m$.
  
  We now prove the ``if'' direction.
  Suppose that $(d_k)_{k=1}^{\infty} \subseteq \nr(T)$ is a sequence with
  \begin{equation*}
    \sum_{k=1}^{\infty} \dist \{ d_k, \mathbb{R} \setminus \essnr(T) \} = \infty,
  \end{equation*}
  and the number of occurrences of each of $a,b$ in the sequence $(d_k)_{k=1}^{\infty}$ is less than or equal to the dimension of the corresponding eigenspace of $T$.
  Let the numbers (possibly $\infty$) of these occurrences be $m,n$ for $a,b$, respectively.
  Then there are orthonormal collections $\{f_k\}_{k=1}^m, \{g_k\}_{k=1}^n$ of eigenvectors for the eigenvalues $a,b$, respectively.
  Note that if $m$ (or $n$) is infinite, we can choose the vectors $f_k$ (or $g_k$) so that the complement of their span inside the eigenspace for $a$ (or $b$) is infinite dimensional.

  Let $P$ be the projection (which commutes with $T$) onto the closed span of $\{f_k\}_{k=1}^m \cup \{g_k\}_{k=1}^n$.
  Note that $P^{\perp}TP^{\perp} = TP^{\perp}$ is a selfadjoint operator with $\essnr(TP^{\perp}) = [a,b]$.
  To see $a$ is still in this set, note that if $m$ is finite, we only removed a finite rank portion of $T$ at $a$.
  Whereas, if $m$ is infinite, our choice of $\{f_k\}_{k=1}^m$ guarantees that the eigenspace of $a$ for $TP^{\perp}$ is infinite dimensional.
  Similar arguments hold for $b$.

  Now, consider the sequence $(d'_k)_{k=1}^{\infty} \subseteq (a,b) = \Int_{\mathbb{R}} \essnr(T)$ obtained by deleting all occurrences of $a,b$ from $(d_k)_{k=1}^{\infty}$.
  We know that
  \begin{equation*}
    \sum_{k=1}^{\infty} \dist \{ d'_k, \mathbb{R} \setminus \essnr(TP^{\perp}) \} = \sum_{k=1}^{\infty} \dist \{ d_k, \mathbb{R} \setminus \essnr(T) \} = \infty,
  \end{equation*}
  because the terms removed from the sequence $(d_k)_{k=1}^{\infty}$ were at distance zero from $\mathbb{R} \setminus \essnr(TP^{\perp})$.
  This ensures the sequence $(d'_k)_{k=1}^{\infty}$ is actually an infinite sequence.
  Therefore by \Cref{thm:blaschke-muller-tomilov}, $(d'_k)_{k=1}^{\infty} \in \mathcal{D}(TP^{\perp})$ (where here the operator acts on the Hilbert space $P^{\perp} \Hil$).
  Combining the orthonormal basis yielding $(d'_k)_{k=1}^{\infty}$ with the collection $\{f_k\}_{k=1}^m \cup \{g_k\}_{k=1}^n$, we obtain that $(d_k)_{k=1}^{\infty} \in \mathcal{D}(T)$.
\end{proof}

\section{Normal operators}
\label{sec:normal-operators}

Up until this point we have only discussed diagonals of selfadjoint operators.
That is primarily because the bulk of the research is focused on this case.
In this section, we explore results on diagonals of normal operators as well.

The first and most natural situation to consider is normal matrices, just as the Schur--Horn theorem (\Cref{thm:schur-horn}) investigated selfadjoint matrices.
In this regard, there are several things to be said.
Upon examination of \Cref{thm:schur-horn}, there are given two characterizations for diagonals of selfadjoint matrices: one in terms of majorization and the other in terms of permutations of the eigenvalue sequence.
Of course, the former is only defined for real sequences, so there is little hope to generalize it to the setting of normal matrices, but the latter is a set of sequences which is easily defined for normal matrices.
One might naively hope that the equivalence of \ref{item:schur-horn-diagonal} and \ref{item:schur-horn-convexity} from \Cref{thm:schur-horn} extends verbatim to the setting of normal matrices.
The following example, due to Horn \cite[pg. 625]{Hor-1954-AJM} based on an idea he attributes to A.~J.~Hoffman, shows that this is not necessarily the case when three of the eigenvalues are not collinear.

\begin{example}
  \label{ex:nonconvex-normal-diagonal}
  Suppose $N$ is a $3 \times 3$ normal matrix with eigenvalues $\lambda_1,\lambda_2,\lambda_3$ which are not collinear.
  Then the sequence with values $d_1 = \frac{\lambda_2 + \lambda_3}{2}, d_2 = \frac{\lambda_1 + \lambda_3}{2}, d_3 = \frac{\lambda_1 + \lambda_2}{2}$ is not a diagonal of $N$.
  To see this, notice that $d$ is a diagonal of $N$ if and only if there is unitary matrix $U$ such that $U(\diag \lambda)U^{*}$ has diagonal $d$.
  Upon computation of the diagonal entries of $U(\diag \lambda)U^{*}$ in terms of the entries of $U$,
  \begin{equation*}
    d_i = \abs{u_{i1}}^2 \lambda_1 + \abs{u_{i2}}^2 \lambda_2 + \abs{u_{i3}}^2 \lambda_3,
  \end{equation*}
  which is a convex combination of $\lambda_1,\lambda_2,\lambda_3$ since the rows of $U$ have norm $1$.
  Since there is a unique way to write each entry $d_i$ as a convex combination $\lambda_1,\lambda_2,\lambda_3$, it is clear that  the square of the moduli of the entries of $U$ must have the form
  \begin{equation*}
    \frac{1}{2}
    \begin{pmatrix}
      0 & 1 & 1 \\
      1 & 0 & 1 \\
      1 & 1 & 0 \\
    \end{pmatrix}
  \end{equation*}
  Then we observe that it is impossible for $U$ to have orthonormal rows/columns, and therefore $U$ cannot actually be unitary.
  Therefore $d$ is not a diagonal of $N$.
\end{example}

In 1973, J.P.~Williams \cite{Wil-1971-JLMS2} completely characterized the diagonals of $3 \times 3$ normal matrices and the description given is entirely geometric.
Below, we restate Williams theorem using the terminology of geometry because it becomes quite concise.
In comparison, the description given by Williams is rather cumbersome.
But first we illustrate with a picture the general case when a diagonal entry does not lie on the boundary of the numerical range described in \Cref{thm:3x3-normal}\ref{item:williams-generic}.

\begin{figure}[h!]
  \centering
  \includegraphics[width=\textwidth]{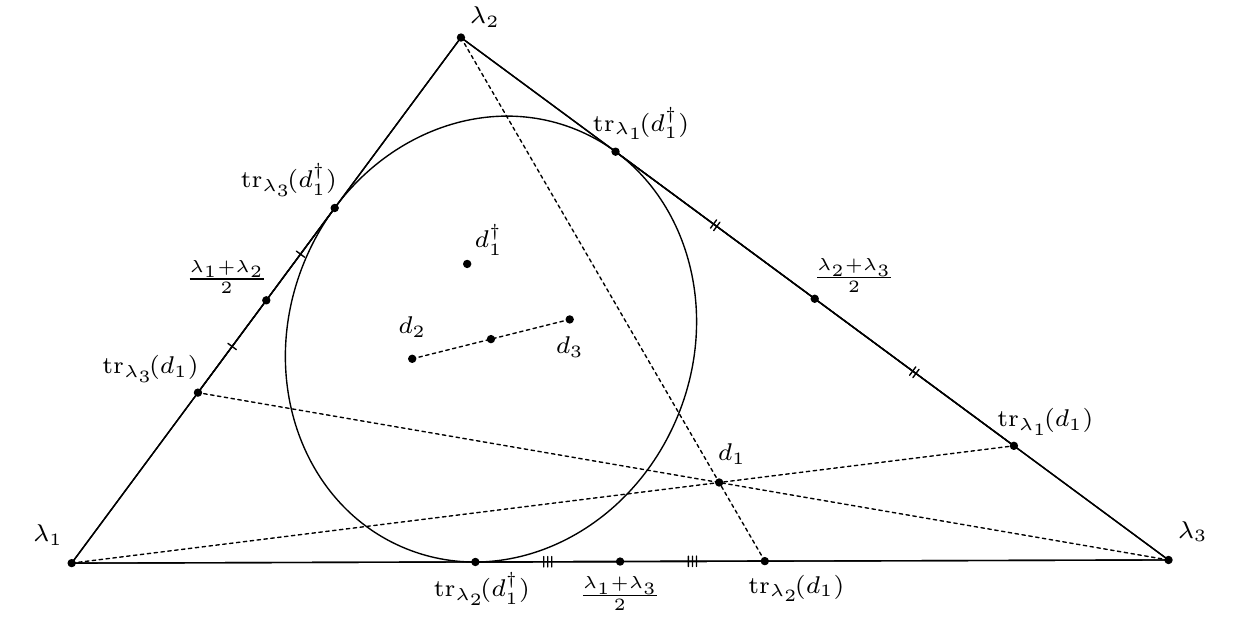}
  \caption{The ellipse described in \Cref{thm:3x3-normal}.
    The point $d_1^{\dag}$ is the isotomic conjugate of $d_1$, and for each $1 \le i \le 3$, $\trace_{\lambda_i}(d_1)$ denotes the trace of $d_1$ on the edge opposite the vertex $\lambda_i$.
    Likewise for $\trace_{\lambda_i}(d_1^{\dag})$.}
  \label{fig:williams-ellipse}
\end{figure}

\begin{theorem}[\protect{\cite[Theorem~2]{Wil-1971-JLMS2}}]\label{thm:3x3-normal}
  Let $N$ be a $3 \times 3$ normal matrix with noncollinear eigenvalues $\lambda_1, \lambda_2, \lambda_3$, so that $W(N) = \conv \{ \lambda_1, \lambda_2, \lambda_3 \}$ is the triangle whose vertices are the eigenvalues.
  Then $(d_1,d_2,d_3) \in \D(N)$ if and only if any of the following mutually exclusive conditions holds:
  \begin{enumerate}
  \item $d_1 = \lambda_i$ and $d_2,d_3$ lie on the edge opposite $\lambda_i$ and they are symmetric about its midpoint.
  \item $d_1 \in \partial W(N) \setminus \{\lambda_1,\lambda_2,\lambda_3\}$ and if $d'_1$ denotes the point which is symmetric to $d_1$ relative to the midpoint of the edge containing $d_1$, then $d_2, d_3$ lie on the line segment joining $d_1$ to $d'_1$ and they are symmetric about its midpoint.
  \item \label{item:williams-generic} $d_1 \in \Int(W(N))$ and $d_2, d_3$ lie in the ellipse inscribed in $W(N)$ which is tangent to $W(N)$ at the traces of isotomic conjugate $d_1^{\dag}$ of $d_1$, and moreover $d_2, d_3$ are symmetric relative to the center of this ellipse (see \Cref{fig:williams-ellipse} for a diagram).
  \end{enumerate}
\end{theorem}

Note that the case when the eigenvalues of $N$ are collinear is actually addressed by \Cref{thm:schur-horn}.
Indeed, if the eigenvalues of $N$ are collinear then there is a selfadjoint matrix $T$ and constant $a,b$ so that $N = aT + b$ and hence $\D(N) = \D(aT + b) = a\D(T) + b$.

One might hope to generalize Williams' result to matrices which are larger than $3 \times 3$.
However, it seems that at this stage there is little hope for progress in this area.
Williams proof of \Cref{thm:3x3-normal} depends in an essential way on the fact that the diagonals of $2 \times 2$ matrices, not just the normal ones, are completely characterized.
Indeed, a diagonal of a $2 \times 2$ matrix $T$ is necessarily of the form $(d, \trace T - d)$ for any $d \in \nr(T)$, and it is well-known that $\nr(T)$ is an ellipse whose foci are the eigenvalues of $T$.
In contrast, there is no such analogous characterization for the diagonal of an arbitrary $3 \times 3$ matrix, which makes the $4 \times 4$ normal case intractable using his approach.

Although necessary and sufficient conditions for a sequence to be a diagonal of a normal operator seem out of reach in general, in \cite{Arv-2007-PNASU} Arveson did manage to determine a necessary condition on diagonals of certain finite spectrum normal operators.
Arveson discovered this condition as a generalization of Kadison's \Cref{thm:kadison-carpenter-pythagorean}.

\begin{theorem}[\protect{\cite[Theorem~4]{Arv-2007-PNASU}}]
  \label{thm:arveson}
  Let $X = \{\lambda_1,\ldots,\lambda_m\}$ be the set of vertices of a convex polygon $P \subseteq \mathbb{C}$ and let $d = (d_1,d_2,\ldots)$ be a sequence of complex numbers satisfying $d_n \in P$ for $n \ge 1$, together with the summability condition
  \begin{equation}
    \label{eq:summability-condition}
    \sum_{n=1}^{\infty} \dist(d_n, X) < \infty.
  \end{equation}
  Then $d$ is the diagonal of a normal operator $N$ with spectrum $\spec(N) = \essspec(N) = X$ if and only if for any $x_n \in X$ such that $\sum_{n=1}^{\infty} \abs{d_n - x_n} < \infty$ there are integers $c_1,\ldots,c_m$ whose sum is zero for which
  \begin{equation*}
    \sum_{n=1}^{\infty} (d_n - x_n) = \sum_{n=1}^m c_n \lambda_n.
  \end{equation*}
\end{theorem}

The above \Cref{thm:arveson} contains \Cref{thm:kadison-carpenter-pythagorean}(ii) as the special case $X = \{0,1\}$.
Moreover, just like \autoref{thm:kadison-carpenter-pythagorean} could be expressed in terms of essential codimension in \Cref{thm:kadison-integer-essential-codimension}, the first author in \cite{Lor-2019-JOT} showed it is possible to do the same for Arveson's theorem.
What follows is a slight generalization and reinterpretation of Arveson's theorem in operator theoretic language.

\begin{theorem}[\protect{\cite[Theorem~4.3]{Lor-2019-JOT}}]
  \label{thm:loreaux-normal-essential-codimension}
  Let $N$ be a normal operator with finite spectrum.
  If $N$ is diagonalizable by a unitary which is a Hilbert--Schmidt perturbation of the identity,
  then there is a diagonal operator $N'$ with $\spec(N') \subseteq \spec(N)$ for which $E(N-N')$ is trace-class.
  Moreover, for any such $N'$,
  \begin{equation}
    \label{eq:trace-in-K-spec-N}
    \trace\big(E(N-N')\big) = \sum_{\lambda \in \spec(N)} [P_{\lambda}:Q_{\lambda}] \lambda,
  \end{equation}
  where $P_{\lambda},Q_{\lambda}$ are the spectral projections onto $\{\lambda\}$ of $N,N'$ respectively.
  Moreover, $P_{\lambda}-Q_{\lambda}$ is Hilbert--Schmidt for each $\lambda \in \spec(N)$.
\end{theorem}

Even though the problem of characterizing diagonals of a \emph{specific} normal operator seems to be intractable at this time, there has been progress determining the diagonals of all normal operators in certain classes.
For example, Horn originally proved \Cref{thm:schur-horn} primarily as a stepping stone to get as his main results \cite[Theorems~8--11]{Hor-1954-AJM} which are characterizations of those finite sequences which are diagonals of some rotation or some unitary matrix.

\begin{theorem}[\protect{\cite{Hor-1954-AJM}}]
  \label{thm:horn-rotation}
  For a sequence $d \in \mathbb{R}^n$, the following are equivalent.
  \begin{enumerate}
  \item $d$ is the diagonal of a rotation matrix, i.e., an orthogonal matrix with determinant 1.
  \item $d \in \conv \{ \lambda \in \{-1,1\}^n \mid \lambda_1 \cdots \lambda_n = 1 \}.$
  \end{enumerate}
  If, in addition, $d \ge 0$, then these are also equivalent to
  \begin{enumerate}[resume]
  \item $d \in \conv \{ \lambda \in \{0,1\}^n \mid \lambda_1 \cdots \lambda_n = 0 \}.$
  \item \label{item:rotation-thompson} $d \in [0,1]^n$ and $2 \left(1 - \min_{1 \le i \le n} d_i \right) \le \sum_{i=1}^n (1 - d_i).$
  \end{enumerate}
\end{theorem}

\begin{theorem}[\protect{\cite{Hor-1954-AJM}}]
  \label{thm:horn-unitary}
  A sequence $d \in \mathbb{C}^n$ (respectively $\mathbb{R}^n$), is the diagonal of an $n \times n$ unitary (respectively, orthogonal) matrix if and only if its sequence of absolute values satisfies any of the equivalent conditions of \Cref{thm:horn-rotation}.
\end{theorem}

Because the unitary (orthogonal) matrices are precisely the matrices whose sequence of singular values is the constant sequence with value 1, it is possible to recognize Theorems \ref{thm:horn-rotation} and \ref{thm:horn-unitary} as special cases of Thompson's Theorem (see \Cref{thm:thompson}) via the equivalence \Cref{thm:horn-rotation}\ref{item:rotation-thompson}.

In \cite{JLW-2016-IUMJ}, the authors with J.~Jasper were able to extend Horn's theorem to the infinite dimensional setting.

\begin{theorem}[\protect{\cite[Theorem~4.3]{JLW-2016-IUMJ}}]
  \label{thm:jasper-loreaux-weiss-unitary}
  A complex-valued sequence $d = (d_n)_{n=1}^{\infty}$ is a diagonal of some unitary operator $U$ if and only if its sequence of absolute values $\abs{d}$ takes values in $[0,1]$ and
  \begin{equation}
    \label{eq:unitary-thompson}
    2 \left(1 - \inf_{n \in \mathbb{N}} d_n \right) \le \sum_{n=1}^\infty (1 - d_n).
  \end{equation}
  Moreover, if $d$ is real-valued then $U$ can be chosen to be orthogonal.
\end{theorem}

\section{General operators}
\label{sec:miscellaneous}

In this section we review results concerning diagonals of general operators with no assumptions of selfadjointness or normality.
One of the early results along these lines is due to Thompson \cite{Tho-1977-SJAM} and, in dimension 2, independently to Sing \cite{Sin-1976-CMB}.
It is a finite dimensional result which characterizes diagonals of the collection of matrices with specified singular values.

\begin{theorem}[\cite{Tho-1977-SJAM}]
  \label{thm:thompson}
  Let $0 \le s \in \mathbb{R}^n$ be a nonincreasing sequence and $d \in \mathbb{C}^n$ a complex-valued sequence.
  There is an $n \times n$ matrix $T$ with singular value sequence $s$ and diagonal $d$ if and only if for the monotone nonincreasing rearrangement $\abs{d}^{*}$ of the sequence of absolute values of $d$,
  \begin{equation*}
    \sum_{i=1}^k \abs{d}_i^{*} \leq \sum_{i=1}^k s_i \quad\text{for } 1 \le k \le n,
  \end{equation*}
  and
  \begin{equation*}
    2(s_n - \abs{d}_n^{*}) \le \sum_{i=1}^n s_i - d_i.
  \end{equation*}
  Moreover, if $d$ is real-valued, we may choose the matrix $T$ to have real-valued entries.
\end{theorem}

In \cite{JLW-2016-IUMJ}, the authors with J.~Jasper were able to extend Thompson's Theorem to compact operators in the natural way.
In particular, since both the sequence $s,d$ converge to zero for a compact operator $T$, it is natural to expect that the second condition in Thompson's theorem disappears entirely, and this is exactly the outcome.

\begin{theorem}[\protect{\cite[Theorem~3.9]{JLW-2016-IUMJ}}]]
  \label{thm:compact-thompson}
  If $s = (s_n)_{n=1}^{\infty}$ is a nonnegative nonincreasing sequence and $d = (d_n)_{n=1}^{\infty}$ is a complex-valued sequence, both converging to zero, then there is a compact operator $T$ with singular value sequence $s$ and diagonal $d$ if and only if, for the monotone nonincreasing rearrangement $\abs{d}^{*}$ of the sequence of absolute values of $d$,
  \begin{equation*}
    \sum_{i=1}^k \abs{d}_i^{*} \leq \sum_{i=1}^k s_i \quad\text{for } k \in \mathbb{N}.
  \end{equation*}
  Moreover, if $d$ is real-valued, we may choose the matrix $T$ to have real-valued entries.
\end{theorem}

As described in \Cref{sec:general-selfadjoint}, the paper \cite{MT-2019-TAMS} of M\"uller and Tomilov focuses on Blaschke-type conditions as sufficient conditions for a sequence to be the diagonal of an operator.
Not only are their results generally applicable for tuples of selfadjoint operators, they also describe several results for general single operators (or tuples).
In fact, \Cref{thm:blaschke-muller-tomilov} for selfadjoint operators extends nearly verbatim to general operators with the key difference being that in \autoref{thm:blaschke-muller-tomilov} the interior of the essential numerical range is taken with respect to $\mathbb{R}$, whereas in \autoref{thm:general-blaschke} it is with respect to $\mathbb{C}$.

\begin{theorem}[\protect{\cite[Corollary 4.3 (simplified to a single operator)]{MT-2019-TAMS}}]
  \label{thm:general-blaschke}
  Let $T \in B(\Hil)$ be any operator and suppose $(d_n)_{n=1}^{\infty}$ is a sequence which takes values in $\Int \essnr(T)$ and satisfies
  \begin{equation}
    \label{eq:blaschke-condition-general}
    \sum_{n=1}^{\infty} \dist ( d_n , \mathbb{C} \setminus \essnr(T) ) = \infty.
  \end{equation}
  Then $(d_n)_{n=1}^{\infty} \in \D(T)$.
\end{theorem}

M\"uller and Tomilov also manage to prove an approximation theorem for diagonals when the Blaschke condition is finite instead of infinite.

\begin{theorem}[\protect{\cite[Corollary 5.1 (simplified to a single operator)]{MT-2019-TAMS}}]
  \label{thm:blascke-p-summable-approximation}
  Let $T \in B(\Hil)$ and let $p > 1$.
  If $(d_n)_{n=1}^{\infty}$ is a complex-valued sequence satisfying
  \begin{equation}
    \label{eq:p-blaschke-summable}
    \sum_{n=1}^{\infty} \dist^p (d_n, \mathbb{C} \setminus \essnr(T)) < \infty, 
  \end{equation}
  then there is a compact operator $K$ in the Schatten ideal $\mathcal{C}_p$ such that $(d_n) \in \D(T+K)$.
\end{theorem}

In addition, it should be noted that the work of Herrero in \cite{Her-1991-RMJM} preceded and inspired \cite{MT-2019-TAMS}, and we have only refrained from mentioning the main theorem of Herrero because it is completely subsumed by \Cref{thm:general-blaschke} above.
In turn, Herrero mentions that he views the start of the attention to diagonals of operators (presumably not just selfadjoint operators) as originating with the work of Fan \cite{Fan-1984-TAMS}, and then later Fan, Fong and Herrero \cite{FFH-1987-PAMS}.

In \cite{Fan-1984-TAMS}, Fan proved an extension to infinite dimensions of the following finite dimensional result: any $n \times n$ matrix with trace zero has an orthonormal basis with respect to which the diagonal is identically zero.

\begin{theorem}[\protect{\cite[Theorem~1]{Fan-1984-TAMS}}]
  \label{thm:fan-zero-diagonal}
  A necessary and sufficient condition that an operator $T$ has a zero-diagonal (i.e., a diagonal whose entries are all zero) is that there exists an orthonormal basis $\{e_j\}_{j=1}^{\infty}$ so that a subsequence of the partial sums
  \begin{equation*}
    \sum_{j=1}^n \angles{Te_j,e_j}
  \end{equation*}
  of the diagonal entries relative to this basis converges to zero.
\end{theorem}

This sparked Fan, Fong and Herrero in \cite{FFH-1987-PAMS} to determine the shape of the set of ``traces'' of an operator $T$.
To explain this, suppose that $\{e_j\}_{j=1}^{\infty}$ is an orthonormal basis relative to which $\sum_{j=1}^n \angles{Te_j,e_j}$ converges to a sum $s$ as $n \to \infty$.
Then $R \{\trace T\}$ denotes the set of all such $s$ as $\{e_j\}_{j=1}^{\infty}$ ranges over all orthonormal bases (for which the partial sums converge).
Fan, Fong and Herrero managed to show that $R \{ \trace T\}$ can only have one of four possible shapes: empty, a point, a line or the entire plane $\mathbb{C}$.
Moreover, they characterized when each situation occurs.

\begin{theorem}[\protect{\cite[Theorem~4]{FFH-1987-PAMS}}]
  \label{thm:fan-fong-herrero-trace-shape}
  For any operator $T \in B(\Hil)$, the set $R \{ \trace T\}$ of traces of $T$ is:
  \begin{enumerate}
  \item empty if and only if for some $\theta$, $\Re(e^{i\theta} T)_+$ is not trace-class, but $\Re(e^{i\theta} T)_+$ \emph{is} trace-class;
  \item a point if and only if $T$ is trace-class;
  \item a line if and only if for some $\theta$, $\Re(e^{i\theta} T)$ is trace-class, but neither of $\Im(e^{i\theta} T)_+$ is trace-class;
  \item the entire plane $\mathbb{C}$ if and only if for all $\theta$, $\Re(e^{i\theta} T)_+$ is not trace-class.
  \end{enumerate}
\end{theorem}

\bibliographystyle{amsalpha}
\bibliography{references}

\end{document}